\documentclass{amsart}

\usepackage[pdftex, colorlinks=true, linkcolor=blue, urlcolor=blue]{hyperref}
\usepackage{amsmath, amsbsy,amsfonts,amsopn, amstext, euscript, cite}   %   amsrefs}
\usepackage{amscd, epic, amssymb,amsxtra, latexsym, euscript} %, cite}
\usepackage{amsmath, amsthm}%,   psfig, epsfig}

\usepackage{amsmath}
\usepackage{amssymb}
\usepackage{amsfonts}

\usepackage{xy} \input xy \xyoption{all}

\newtheorem{theorem}{Theorem}[section]

\newtheorem{lemma}{Lemma}[section]

\newtheorem{proposition}{Proposition}[section]

\newtheorem{example}{Example}[section]

\newtheorem{remark}{Remark}[section]

\def\Z{\mathbb Z}

\def\bP{\mathbb P}

\def\L{\mathcal L}

\def\H{\mathcal H}
\def\M{\mathcal M}
\def\N{\mathcal N}

\def\l{\lambda}
\def\s{\sigma}
\def\ss{\bar \sigma}

\def\a{\alpha}

\def\b{\beta}
\def\p{\mathfrak p}
\def\P{\mathcal W}
\def\e{\varepsilon}
\def\iso{\equiv}

\def\o{\oplus}
\def\g{\gamma}

\def\u{\mathfrak u}

\def\iso{{\, \cong\, }}

\def\<{\langle}
\def\>{\rangle}

\def\Aut{\mbox {Aut}}
\def\bAut{\overline {\mbox {Aut}}}
\def\G{\overline G}

\def\ff{\phi}
\def\f{\psi}
\def\X{\mathcal X}
\def\Y{\mathcal Y}
\def\fH{\mathfrak H}
\def\t{\mu}
\def\d{\delta}
\def\m{\mu}
\def\v{\mathfrak v}
\def\r{\mu}

\title{Some special families of hyperelliptic curves}

\author{Tanush Shaska}

\address{Department of Mathematics, \\
University of Idaho \\
E-mail: tshaska@uidaho.edu }

\date{This paper has been published by Journal of Algebra and Applications on January 2004}


\begin{document}

\begin{abstract}
Let $\L_g^G$ denote the locus of hyperelliptic curves of genus $g$ whose automorphism group contains a subgroup
isomorphic to $G$.  We study spaces $\L_g^G$ for $G \iso \Z_n, \Z_2\o \Z_n, \Z_2\o A_4$, or $SL_2(3)$. We show
that for $G \iso \Z_n, \Z_2\o \Z_n$, the space $\L_g^G$ is a rational variety and find generators of its function
field.  For $G\iso \Z_2\o A_4, SL_2(3)$ we find  a necessary condition in terms of the coefficients, whether or
not the  curve  belongs to $\L_g^G$.  Further,  we  describe algebraically the loci of such curves for $g\leq 12$
and show that  for all curves in these loci  the field of moduli is a field of definition.
\end{abstract}

\maketitle

\section{Introduction}
%&&&&&&&&&&&&&&&&&&&&&&&&&&&&&&&&&&&&&&&&&&&&&&&&&&&&&&&&&&&&&&&&

One of the most interesting     problems     in     algebraic geometry is    to    obtain a generalization of the
theory of elliptic modular functions to the case of higher genus. In the elliptic case this is done by the
so-called $j$-{\it invariant} of elliptic curves. In the case of genus $g=2$,   Igusa (1960) gives a complete
solution via {\it absolute invariants} $i_1, i_2, i_3$ of genus 2 curves; see \cite{Ig}. Generalizing such results
to higher genus is much more difficult due to the existence of non-hyperelliptic curves. However, even restricted
to the hyperelliptic moduli $\fH_g$, the problem is still unsolved for $g \geq 3$. In other words, there is no
known way of identifying isomorphism classes of hyperelliptic curves of genus $g\geq 3$. In terms of classical
invariant theory  this  means  that the  field of invariants of binary forms of degree $2g+2$ is not known for
$g\geq 3$.

In previous work we  have focused on the loci $\L_g^G$  of hyperelliptic curves with $G$  embedded in the
automorphism group.  In \cite{GS}  we introduced  a way (via dihedral invariants) of identifying isomorphism
classes of genus $g$ hyperelliptic curves with non-hyperelliptic involutions.     In this paper we study cases
when the automorphism group is isomorphic to one of the following: $\Z_n, \Z_2\o \Z_n, \Z_2\o A_4$, and $SL_2(3)$.
This is part of a larger project of the author of finding an algorithm which determines the automorphism group of
hyperelliptic curves via their invariants and determining whether the field of moduli is the  same as the  field
of definition; see  \cite{Sh5}.

The second  section covers basic facts  on automorphism groups of hyperelliptic curves, Hurwitz spaces, and
invariants of binary forms. Let  $\X_g$ denote a  genus  $g$  hyperelliptic  curve defined  over an algebraically
closed field $k$  of characteristic zero, $\Aut(\X_g)$ its automorphism  group, and $z$ the hyperelliptic
involution of $\X_g$.  The group $\bAut(\X_g):=\Aut(\X_g)/\<z\>$  is called the {\it reduced automorphism group}
of $\X_g$.   In this paper we study hyperelliptic curves with reduced automorphism group isomorphic to a cyclic
group $\Z_n$ or $A_4$. We determine the ramification signature $\s$ of the cover  $\f : \X_g \to \bP^1$ with
monodromy group $G:=\Aut(\X_g) $ (cf. section 2.2 for details). Hurwitz spaces are moduli spaces of such covers
$\f$ which we denote by $\H_\s$.  There is  a map $$\Phi_\s: \H_\s \to \M_g$$  where $\M_g$ is the moduli space of
genus $g$ algebraic curves.   We denote by $\L_g^G(\s)$ the image $\Phi_\s (\H_\s)$ in the hyperelliptic locus
$\fH_g$.  Given a curve $\X_g$ we would like to determine if it belongs to the locus $\L_g^G(\s)$ and  describe
points $\p \in \L_g^G(\s)$. Hence, in section 2.3 we introduce invariants of binary forms.

In section three,  we study  the case when $\bAut(\X_g)$ is a cyclic group $\Z_n$.  There are three possible
signatures and two types of groups that occur as full automorphism groups, namely $\Z_{2n}$ and $\Z_2 \o \Z_n$. We
show that in all three cases $\L_g^G$ is a rational $\d$-dimensional variety and  $k(\L_g^G)=k(u_1, \dots , u_\d)$
with $\u:=(u_1, \dots , u_\d)$ defined in terms of the coefficients of the curve.  There is a 1-1 correspondence
between nonsingular points of $\L_g^G$ and  tuples  $\u=(u_1, \dots , u_\d)$.

In section four, we focus on the case when $\bAut (\X_g) \iso A_4$. There are six possible signatures  of the
cover  $\f : \X_g \to \bP^1$.  Three of these ramifications have $\Z_2 \o A_4$ as monodromy group and the other
three have  $SL_2(3)$.  We find  equations of these curves in each case.  We prove that if $\bAut (\X_g) \iso
A_4$, then $I_4 ( \X_g )=0$ (cf. section 3).  This gives a nice necessary condition of checking whether $\X_g$ has
automorphism group $\Z_2 \o A_4$ or $SL_2(3)$.

In the last section we focus on zero or 1-dimensional subvarieties $\L_g^G$ of $\fH_g$ for $G\iso \Z_2 \o A_4,
SL_2(3)$. In each case we find explicit  equations of such varieties in terms of $GL_2(k)$-invariants of binary
forms of degree $2g+2$. This gives an efficient algebraic way of determining if the automorphism group of a  genus
$g \leq 12$ is $\Z_2 \o A_4$ or $ SL_2(3)$. Such a method can be easily generalized to higher genus.  Further, we
show that for each $\X_g\in \L_g^G$, $g\leq 12$,  the field of moduli is the same as the field of definition.

\medskip

\noindent  {\bf  Notation:}  Throughout  this  paper  $k$  denotes  an algebraically  closed field  of
characteristic zero,  $g$ an  integer $\geq 2$, and $\X_g$ a  hyperelliptic curve of genus $g$ defined over $k$.
$\fH_g$ is the moduli space  of  hyperelliptic curves defined over  $k$.

%******************************************
\section{Preliminaries}
%*****************************************

In this section we recall some basic facts about hyperelliptic curves and their automorphisms, Hurwitz spaces, and
invariants of binary forms.
%******************************************************************
\subsection{Hyperelliptic curves and their automorphisms}
%******************************************************************

Let $k$  be an algebraically  closed field of characteristic  zero and $\X_g$  be a  genus  $g$  hyperelliptic
curve  given  by  the equation $Y^2=F(X)$, where $\deg(F)=2g+2$. Denote  the function field of $\X_g$ by
$K:=k(X,Y)$.  Then, $k(X)$ is the  unique degree 2 genus zero subfield of  $K$. $K$ is  a quadratic extension
field of $k(X)$ ramified exactly at $d=2g+2$ places $\a_1, \dots , \a_d$  of $k(X)$. The corresponding places  of
$K$ are called the {\it Weierstrass points} of $K$.   Let $\P:=\{ \a_1, \dots , \a_d \}$ and $G=Aut(K/k)$. Since
$k(X)$  is the only  genus 0 subfield  of degree  2  of $K$, then  $G$  fixes $k(X)$.  Thus, $G_0:=Gal(K/k(X))=\<
z_0 \>$, with $z_0^2=1$, is central in $G$. We call  the {\it reduced automorphism group}  of $K$ the  group
$\G:=G/G_0$. We illustrate with the following  diagram:
$$
\xymatrix{ K \ar@{-}[d]^{\, \, \, \Z_2} \ar@/_1.8pc/[dd]_{\, \, G} \\
k(X) \ar@{-}[d]^{\, \, \, \G}  \\ k  \\ }  \qquad \qquad \qquad \xymatrix{ \X_g \ar[d]_{\, \, \, }^{\, \, \, \Z_2}
\ar@/_1.8pc/[dd]_{\, \, \f} \\ \bP^1 \ar[d]_{\, \, \, \ff}^{\, \, \, \G} \\ \bP^1  \\ }
$$
By a theorem of Dickson, $\G$ is isomorphic to one of the following:  $\, \Z_n$, $D_n$, $A_4$, $S_4$, $A_5$ with
branching indices of the  corresponding cover $\bP^1 \to \bP^1/ \G$ given respectively by
$$(n,n), (2, 2, n), (2, 3, 3), (2, 4, 4), (2, 3, 5).$$  We focus on cases $\G \iso \Z_n, A_4$, other cases are
intended to be studied in \cite{Sh6}.

%***************************************************
\subsection{Hurwitz spaces of covers $\f: \X_g \to \bP^1$}
%**************************************************
%
Let $\M_g$ be the moduli space of curves of genus $g \geq 2$ and $\bP^1=\bP^1(k)$ the Riemann sphere.  Let  $\phi:
\X_g \to \bP^1$ be a degree  $n$ covering with $r$ branch points. By covering space theory, there is a tuple
$(\s_1, \dots , \s_r)$ in $S_n$ such that $\s_1 \cdot \cdot \cdot \s_r =1$ and $G:=< \s_1, \dots , \s_r> $ is a
transitive group in $S_n$. We call such a tuple the {\it signature} of $\phi$. We say that a  permutation is of
type  $n^p$ if  it is a product  of $p$ disjoint  $n$-cycles.

Conversely, let $\s: = (\s_1, \dots , \s_r)$ be a tuple in $S_n$ such that $\s_1 \cdot \cdot \cdot \s_r =1$ and
$G:=< \s_1, \dots , \s_r> $ is a transitive group in $S_n$. We say that a cover $\phi: \X \to \bP^1$ of degree
$n$ is of type $\s$ if it has $\s$ as signature. The genus $g$ of $\X$ depends only on $\s$ (Riemann-Hurwitz
formula). Let $\H_\s$ be the set of pairs $([f], (p_1, \dots , p_r)$, where $[f]$ is an equivalence class of
covers of type $\s$, and $p_1, \dots , p_r$ is an ordering of the branch points of $\phi$.  The {\it Hurwitz
space} $\H_\s$ is a quasiprojective variety; see \cite{FrV}. We have a morphism
$$\Phi_\s: \H_\s \to \M_g$$
mapping $([f], (p_1, \dots , p_r)$ to the class $[\X]$ in the moduli space $\M_g$.  Each component of $\H_\s$ has
the same image in $\M_g$.

We denote  by $C:=(C_1, \dots , C_r)$, where $C_i$ is the conjugacy class of $\s_i$ in $G$.  The set of Nielsen
classes $\N(G, C)$ is
$$\N(G, C):=\{ (\s_1, \, \dots \, , \s_r) \, |  \, \,  \s_i \in C_i,
\, G=< \s_1, \dots , \s_r >, \, \,  \s_1 \cdots \s_r=1  \} $$
Fix a base point $\l_0\in \bP^1\setminus S$ where $S$ is the set of branch points. Then $\pi_1(\bP^1\setminus S)$
is generated by homotopy classes of loops $\g_1, \dots , \g_r$. The braid group acts on $\N(G, C)$ as
$$[\g_i]: \quad (\s_1, \, \dots \, , \s_r) \to (\s_1, \, \dots , \,
\s_{i-1},    \s_i \s_{i+1} \s_i^{-1}      , \s_i, \s_{i+2}, \dots , \s_r)
$$
The orbits of this action are called the {\it braid orbits} and correspond to the irreducible components of $\H
(G,C):=\H_\s$.

\begin{lemma} Let $\X_g$ be a genus $g\geq 2$ hyperelliptic curve with
 $\G:= \Z_n, A_4$. Then,  $G:=\Aut(\X_g)$, the dimension $\delta$ of
$\L_G^\s$, the signature $\s$,  and the number of involutions $i(G)$ of $G$ are as follows:
\begin{table} [ht]
\begin{center}
\renewcommand{\arraystretch}{1.24}
\begin{tabular}{||c|c|c|c|c|c|c||}
\hline \hline Case &$G$ &$\G$& $\delta=dim(\L_G^\s)$  & $\delta\neq $ &$\s=(\s_1, \dots , \s_r)$ & i(G)  \\ \hline
\hline 1 &$\Z_2\o \Z_n $& &$\frac {2g+2} n - 1$  & $\delta\neq 0, 1$ & $(n^2, n^2, 2^n, \dots , 2^n)$ & 3 \\ 2
&$\Z_{2n}$ & $\Z_n$&$\frac {2g+1} n -1$   & & $(n^2, 2n, 2^n, \dots , 2^n )$ &  1  \\ 3 & $\Z_{2n}$& &$\frac {2g}
n -1 $ & $\delta\neq 0, 1$ &  $(2n, 2n, 2^n,\dots , 2^n)$   &  1  \\ \hline \hline 4 & $\Z_2\o A_4$ & &$\frac
{g+1} 6$ & & $(3^8, 3^8, 2^{12}, \dots , 2^{12} )$ &    \\ 5 &$\Z_2\o A_4$ & &$\frac {g-1} 6$  &   & $(3^8, 6^4,
2^{12}, \dots , 2^{12} )$ &  7 \\ 6 & $\Z_2\o A_4$&$A_4$&$\frac {g-3} 6$   & $\delta \neq 0$   &$(6^4, 6^4,
2^{12}, \dots , 2^{12})$   &    \\ 7 & $SL_2(3) $ &  & $\frac {g-2} 6$ & $\delta \neq 0$   & $(4^6, 3^8, 3^8,
2^{12}, \dots , 2^{12})$& \\ 8 & $SL_2(3) $ & & $\frac {g-4} 6$           & & $(4^6, 3^8, 6^4, 2^{12}, \dots ,
2^{12})$&1 \\ 9 & $SL_2(3) $ &  & $\frac {g-6} 6$ & $\delta \neq 0$   & $(4^6, 6^4, 6^4, 2^{12}, \dots , 2^{12})$&
\\ \hline \hline
\end{tabular}
\end{center}
\caption{Hyperelliptic curves with reduced automorphism group $\Z_n, A_4$}
\end{table}
\end{lemma}
\begin{proof}
The proof is elementary and follows from results in \cite{Bu}.
\end{proof}

Let $G:=\Z_2 \o A_4, SL_2(3)$. Denote by $\H_G^g$ the Hurwitz space $\H(G,C)$ of covers  given in Table 1.  Spaces
$\H_G^g$ are irreducible. To show this we have to show that there is only one braid orbit. By applying the braid
action one can assume that the signature is given as  $\s=( \s_1, \s_2, \s_3, \a, \dots , \a)$ where $\s_1, \s_2,
\s_3$ are as in Table 1, and $\a$ is an involution. Thus, we are looking for 4-tuples $(\s_1, \s_2, \s_3, \a)$ or
3-tuples $(\s_1, \s_2, \s_3)$, depends whether $r$ is even or odd,  which are transitive  in $S_{24}$ and generate
$G$. A computer search would show that there is only such braid orbit.  For $r=4$ these spaces are genus zero
curves as can be shown by a direct computation. In section 5 we will describe  these spaces algebraically.
\begin{example}Let $g=5$. Then, $\H_G^5$ has genus 0. There are 6
Nielsen classes. One of them is $\s=(\a, \a, \beta, \beta^{-1})$, where $\a, \beta $ are as below:
\begin{small}
\begin{eqnarray*}
\a =  & ( 1, 2)( 3, 4)( 5, 6)( 7, 8)(
9,10)(11,12)(13,14)(15,16)(17,18)(19,20)(21,22)(23,24), \\
\b =  & (1, 2, 3)( 4, 5, 6) ( 7, 8, 9)(10,11,12)(13,14,15)(16,17,18)(19,20,21)(22,23,24).
\end{eqnarray*}
\end{small}
\end{example}

%\vspace{-8mm}

\begin{remark}
An interesting problem would be to decide if $\Phi_\s (\H_\s) \subset \fH_g$?  In other words, are there any genus
$g$ non-hyperelliptic curves $\Y_g$ such that there is a cover $\f : \Y_g \to \bP^1$ with signature as in Table 1.
\end{remark}

%\begin{remark}
For the rest of this paper $\L_g^G(\s)$ is denoted by $\L_g^G$, since for each genus  $g$ there is exactly one
signature $\s$.
%\end{remark}

%************************************************************
\subsection{Invariants of Binary Forms}
%************************************************************
In this section we define the action of $ GL_2(k)$ on binary forms and discuss the basic notions of their
invariants. Let $k[X,Z]$  be the  polynomial ring in  two variables and  let $V_d$ denote  the
$(d+1)$-dimensional  subspace  of  $k[X,Z]$  consisting of homogeneous polynomials.
\begin{equation}  \label{eq1}
f(X,Z) = a_0X^d + a_1X^{d-1}Z + ... + a_dZ^d
\end{equation}
of  degree $d$. Elements  in $V_d$  are called  {\it binary  forms} of degree $d$.  We let $GL_2(k)$ act as a
group of automorphisms on $ k[X, Z] $   as follows:
\begin{equation}
 M =
\begin{pmatrix} a &b \\  c & d
\end{pmatrix}
\in GL_2(k), \textit{   then       }
\quad  M  \begin{pmatrix} X\\ Z \end{pmatrix} =
\begin{pmatrix} aX+bZ\\ cX+dZ \end{pmatrix}
\end{equation}
This action of $GL_2(k)$  leaves $V_d$ invariant and acts irreducibly on $V_d$.
\begin{remark}
It is well  known that $SL_2(k)$ leaves a bilinear  form (unique up to scalar multiples) on $V_d$ invariant. This
form is symmetric if $d$ is even and skew symmetric if $d$ is odd.
\end{remark}
Let $A_0$, $A_1$,  ... , $A_d$ be coordinate  functions on $V_d$. Then the coordinate  ring of $V_d$ can be
identified with $ k[A_0  , ... , A_d] $. For $I \in k[A_0, ... , A_d]$ and $M \in GL_2(k)$, define $I^M \in k[A_0,
... ,A_d]$ as follows
\begin{equation} \label{eq_I}
{I^M}(f):= I(M(f))
\end{equation}
for all $f \in V_d$. Then  $I^{MN} = (I^{M})^{N}$ and Eq.~(\ref{eq_I}) defines an action of $GL_2(k)$ on $k[A_0,
... ,A_d]$.
A homogeneous polynomial $I\in k[A_0, \dots , A_d, X, Z]$ is called a {\it covariant}  of index $s$ if
$$I^M(f)=\delta^s I(f),$$
where $\delta =\det(M)$.  The homogeneous degree in $a_1, \dots , a_n$ is called the {\it degree} of $I$,  and the
homogeneous degree in $X, Z$ is called the {\it  order} of $I$.  A covariant of order zero is called {\it
invariant}.  An invariant is a $SL_2(k)$-invariant on $V_d$.

We will use the symbolic method of classical theory to construct covariants of binary forms.    Let
$$f(X,Z):=\sum_{i=0}^n
\begin{pmatrix} n \\ i
\end{pmatrix}
a_i X^{n-i} \, Z^i, \quad  and \quad g(X,Z) :=\sum_{i=0}^m
  \begin{pmatrix} m \\ i
\end{pmatrix}
b_i X^{n-i} \, Z^i
$$
be binary forms of  degree $n$ and $m$ respectively with coefficients in $k$. We define the {\bf r-transvection}
$$(f,g)^r:= \frac {(m-r)! \, (n-r)!} {n! \, m!} \, \, \sum_{k=0}^r
(-1)^k
\begin{pmatrix} r \\ k
\end{pmatrix} \cdot
\frac {\partial^r f} {\partial X^{r-k} \, \,  \partial Z^k} \cdot \frac {\partial^r g} {\partial X^k  \, \,
\partial Z^{r-k} }
 $$
 It is a homogeneous polynomial in $k[X, Z]$ and therefore a covariant
of order $m+n-2r$ and degree 2. In general, the $r$-transvection of two covariants of order $m, n$ (resp., degree
$p, q$) is a covariant of order $m+n-2r$  (resp., degree $p+q$).

For the rest of this paper $F(X,Z)$ denotes a binary form of order $d:=2g+2$ as below
\begin{equation}
F(X,Z) =   \sum_{i=0}^d  a_i X^i Z^{d-i} = \sum_{i=0}^d
\begin{pmatrix} n \\ i
\end{pmatrix}    b_i X^i Z^{n-i}
\end{equation}
where $b_i=\frac {(n-i)! \, \, i!} {n!} \cdot a_i$,  for $i=0, \dots , d$.  We denote invariants (resp.,
covariants) of binary forms by $I_s$ (resp., $J_s$) where the subscript $s$ denotes the degree (resp., the order).
We define the following covariants and invariants:
\begin{equation}
\begin{split}\label{covar}
\aligned
 I_2 & :=(F,F)^d,   \\
 I_4 & :=(J_4, J_4)^4,  \\
 I_6 & :=((F, J_4)^4, (F, J_4)^4)^{d-4},   \\
 I_6^\ast & :=((F, J_{12})^{12}, (F, J_{12})^{12})^{d-12},  \\
  M  & :=((F, J_4)^4, (F, J_8)^8)^{d-10}, \\
\endaligned
\qquad \aligned
& J_{4j}   :=   (F,F)^{d-2j}, \, \,  j=1, \dots , g, \\
& I_4'    :=     (J_8, J_8)^8, \\
& I_6^\prime  :=((F, J_8)^8, (F, J_8)^8)^{d-8}, \\
& I_3      :=(F, J_d)^d, \\
& I_{12}   :=(M, M)^8\\
\endaligned
\end{split}
\end{equation}
{\it Absolute invariants} are called $GL_2(k)$-invariants. We  define the following absolute invariants:
$$i_1:=\frac {I_4'} {I_2^2}, \, \,  i_2:=\frac {I_3^2} {I_2^3},
\, \,  i_3:=\frac {I_6^\ast  } {I_2^3}, \, \,  j_1 := \frac {I_6^{'}} {I_3^2}, \, \,   j_2:= \frac {I_6} {I_3^2},
\, \, s_1:=\frac {I_6^2} {I_{12}}, \, \, s_2:=\frac {(I_6^{'})^2} {I_{12}}
$$
$$
\v_1:= \frac {I_6} {I_6^\ast }, \, \,  \v_2:=\frac {(I_4^{'})^3} {I_3^4}, \, \, \v_3:= \frac {I_6} {I_6^{'}}, \,
\,  \v_4:=\frac {(I_6^\ast )^2} {I_4^3}.
$$
In the case $g=10$ and $I_{12}=0$ we define
\begin{equation}
\begin{split}
 I_6^\star & := ( (F, J_{16})^{16},  (F, J_{16})^{16} )^{d-16}), \\
S   \, \, & :=( J_{12}, J_{16} )^{12}, \\
I_{12}^\ast & := ( \, (J_{16}, S)^4, \, (J_{16}, S)^4 \, )^{12}\\
\end{split}
\end{equation}
and $$\v_5\,  :=\frac {I_6^\star }{I_{12}^\ast}.$$

\noindent For a given curve $\X_g$ we denote by $I(\X_g)$ or $i(\X_g)$ the corresponding  invariants.
\begin{remark}
We will only perform computations on subvarieties $\L_g^G\subset \fH_g$ of  dimension $\delta \leq 1$,  hence
don't need other absolute invariants.
\end{remark}

%*************************************************
\section{The reduced automorphism group is  cyclic}
%*************************************************
%
Let $\bAut (\X_g)\iso \Z_n$.   Then, $\ff: \bP^1 \to \bP^1$ has
 signature $(n, n)$. We identify the branch points of $\ff$ with $0,
 \infty$ and the ramified points in their fibers by $0, \infty$
 respectively. Hence, $\ff(X) = X^n$.  We denote by $V:=\ff^{-1}(0)
 \cup \ff^{-1}(\infty)$. In this section  $e_t$ denotes the $t$-th
root of unity, $G$ and $\d$ are as in first three cases of Table 1.

\medskip
%*****************************************************
\noindent {\bf Case 1:}
%************************************************
If $V \cap \P =\emptyset$,  then $n\, | \, 2g+2$ and the equation of the curve is
$$Y^2= \prod_{i=1}^t (X^n-q_i)$$
where $q_i$'s are the branch points of $\f : \X_g \to \bP^1$ not in $\{0, \infty\}$ and $t = \frac {2g+2} n$.  Let
$a_1, \dots , a_t$ denote the symmetric polynomials in $q_1, \dots , q_t$. Further we can take $q_1 \dots q_t=1$.
Hence the equation of the curves is
\begin{equation}
Y^2= X^{2g+2} + a_1 X^{n \, (t-1)} + \dots + a_i X^{n\, (t-i)} + \dots + a_{\d} X^n + 1.
\end{equation}
 We need to determine to what extent the normalization   above
determines the coordinate $X$. Let $\g$ generate $\Z_n$. Then, $\g(X)=e_n X$. This  condition  determines the
coordinate $X$ up to a coordinate change by some $\a \in PGL_2(k)$ centralizing $\g$. Such $\a$ satisfies $\a
(X)=mX$ or $\a (X) = \frac m X$, $m \in k\setminus \{0\}$. The additional condition $(-1)^t q_1 \cdot \cdot \cdot
\cdot q_t=1$ forces
$$(-1)^t \, \g(q_1) \dots \g(q_t)=1.$$
Hence,   $m^t=1$. So $X$ is determined up to a coordinate change by the subgroup $H\iso D_{2t}$ of $PGL_2(k)$
generated by $\tau_1: X\to \e_t X$, $\tau_2: X\to \frac  1 X$, where $\e_t$ is a primitive $t$-th root of unity.

%********************************************************
\medskip
\noindent {\bf Case 2:}
%********************************************************
If $|V \cap \P|=1$ then 0 or $\infty$ is a Weierstrass point. Then, $n \, | \, 2g+1$  and the equation of the
curve is
$$Y^2=  \prod_{i=1}^t (X^n-q_i)$$
where $q_i$'s are the branch points of $\f : \X_g \to \bP^1$ and $t = \frac {2g+1} n$. Hence the equation of the
curve is
\begin{equation}
Y^2=  X^{2g+1} + a_1 X^{n \, (t-1)} + \dots + a_{\d} X^n + 1
\end{equation}
Then $X$ is determined up to a coordinate change by the subgroup
 $H:=\< \tau_1, \tau_2 \>$  of $PGL_2(k)$ such that $\tau_1: X\to \e_t X$,
$\tau_2: X \to \frac 1 X$.
%

%********************************************************
\medskip
\noindent {\bf Case 3:}
%********************************************************
If $|V \cap \P|=2$ then $n\, | \, 2g$ and the curve has equation
\begin{equation}
Y^2=X(X^{nt}+ a_1 X^{n(t-1)} + \dots +  a_{\d} X^n +1)
\end{equation}
where $t = \frac {2g} n$.  Then $X$ is determined up to a coordinate change by the subgroup $H:=\< \tau_1, \tau_2
\>$ of $PGL_2(k)$ such that $\tau: X\to \e_t X$, $\tau_2: X\to \frac 1 X$.

\bigskip

\noindent Now we consider all three cases.  $H$ acts on $k(a_1, \dots , a_\d )$ as follows:
\begin{eqnarray*}
 \tau_1: &\quad   a_i \to \e^{d-ni} a_i,  & \quad for \quad i=1, \dots , \d  \\
 \tau_2: &\quad  a_i \to a_{t-i},        &  \quad for \quad i=1, \dots , [\frac {\d + 1} 2]
\end{eqnarray*}
Thus, the fixed field $k(a_1, \dots , a_{\d})^H$ is the same as the function field of the variety $\L_g^G$. We
summarize in  the following:

\begin{lemma}   $k(\L_g^G)=k(a_1, \dots , a_{\d})^H$.
\end{lemma}

\noindent The following
\begin{equation}\label{u_i}
u_i:=  a_1^{t-i} \, a_i \, + \, a_{\d}^{t-i} \, a_{t-i}, \quad for
\quad 1 \leq i \leq \d\\
\end{equation}
are called {\it dihedral invariants} for the genus  $g$ and the tuple
$$\u:=(u_1, \dots , u_\d)$$ is called the {\it tuple of dihedral
invariants}.  It can be checked that  $\u=0$ if and only if $a_1=a_\d=0$. In this case replacing $a_1, a_\d$ by
$a_2, a_{\d-1}$ in the formula above would give new invariants. We would focus in the case that $\u \neq 0$, as
the other cases are simpler.
The next theorem shows that the dihedral invariants generate $k(\L_g^G)$.
\begin{theorem}
Let $ \L_g^G$  be as in cases  1, 2, 3, of Lemma 2.1. and $\d=\dim(\L_g^G)$. Then, $k(\L_g^G) =k(\u_1, \dots ,
\u_\d)$.
\end{theorem}

\begin{proof}The dihedral invariants are fixed by the $H$-action. Let
$\u=(u_1, \dots , u_{\d})$ be the  $\d$-tuple of dihedral invariants. Hence, $k(\u)\subset k(\L_g^G) $. Thus, it
is enough to show that $[k(a_1, \dots a_{\d}): k(\u)]=2t$. For each $2 \leq i \leq \d-1$ we have
\begin{eqnarray*}
& a_1^{\d-i+1} a_i + a_{\d}^{\d-i+1} a_{\d-i+1} = u_i\\
& a_1^i a_{\d-i+1} + a_{\d}^i a_i=u_{\d-i+1}
\end{eqnarray*}
giving $\, a_i, \, \, a_{t-i} \in k(\u, a_1, a_{\d})$. Then, the extension $k(a_1, \dots , a_{\d}) /  k(u_1, \dots
, u_{\d})$ has equation
\begin{equation}
2^{t}\, a_g^{2t } - 2^{t}\, u_1 \,  a_{\d}^{t} + u_{t-1}^{t}=0
\end{equation}
This  completes the proof.

\end{proof}
\begin{remark}
If $n=2$ then $G=V_4$. Then  $\L_g^G=\L_g$ where $\L_g$ is the locus of hyperelliptic curves with extra
involutions; see \cite{GS}.  A nice necessary and sufficient condition is found in \cite{Sh5} in terms of the
dihedral invariants for a curve to have more then three involutions in the reduced automorphism group. More
precisely, for such curves the relation $2^{g-1} u_1^2 - u_g^{g+1}=0$ holds.
\end{remark}
%
%
%***************************************
\section{The reduced automorphism group is isomorphic to $A_4$}
%**************************************
%
In this section we study genus $g$ hyperelliptic curves $\X_g$ with $\bAut(\X_g)\iso A_4$.  Thus, $A_4$ is the
monodromy group of a cover $\ff: \bP^1 \to \bP^1$ with signature $\ss:=(\s_1, \s_2, \s_3)$ of type $(2^6, 3^4,
3^4)$; see Table 1.  We denote by $q_1, q_2, q_3 $ the corresponding branch points of $\ff$.  Let $S$ be the set
of branch points of $\f: \X_g \to \bP^1$. Clearly $q_1, q_2, q_3 \in S$. As above  $\P$ denotes the images in
$\bP^1$ of Weierstrass points of $\X_g$ and $V:=\cup_{i=1}^3 \ff^{-1} (q_i) $.
\begin{lemma} % \label{lem}
Let $V, \P$ be as above and $g \neq 2, 3, 6$. Then the following hold:

i) if $\, \, |V \cap \P|=0, 4, 8, $ then  $\Aut(\X_g) \iso \Z_2 \o
 A_4$ and  $g\equiv -1, 1, 3 \mod 6$ respectively,

ii) if $\, \, |V \cap \P|=6, 10, 14, $ then $\Aut(\X_g) \iso SL_2(3)$ and  $g \equiv 2, 4, 0 \mod 6$ respectively.
\end{lemma}

\begin{proof}
Assume that $|V \cap \P|=0$.   Then, $\ff(w_i)$ are branch points of $\ff$. By Riemann-Hurwitz formula $2g+2\equiv
0 \mod 12$. Then, $g\equiv -1 \mod 6$. The number of branch points of $\f : \X_g \to \bP^1$ is $r=3 + \frac {g+1}
6$. If $|V \cap \P| = 4$ then either $\ff^{-1} (q_2) \subset \P$ or $\ff^{-1} (q_3) \subset \P$. Hence, $2g-2
\equiv 0 \mod 12$ or $g \equiv 1 \mod 6$. If $|V \cap \P| = 8$ then $\ff^{-1} (q_i) \subset \P$ for $i=2,3$. Thus,
$2g-6 \equiv 0 \mod 12$ or $g \equiv 3 \mod 6$. In all cases $\s_1$ lifts to a non hyperelliptic involution in
$G$. Hence $G$ has more then one involution. Hence, $\Aut(\X_g) \iso  \Z_2 \o A_4$.

If $|V \cap \P| = 6, 10, 14 $ then  $\ff^{-1} (q_1) \subset \P$, $\ff^{-1} (q_i) \subset \P$ for $i=1,2$,
$\ff^{-1} (q_i) \subset \P$ for $i=1, 2, 3$ respectively. Thus, $g \equiv 2, 4, 0 \mod 6$. In all cases $\s_1$
lifts to an element of order 4 in $\Aut(\X_g)$. Thus, $\Aut(\X_g)$ has only the hyperelliptic involution. Hence,
$\Aut(\X_g) \iso SL_2(3)$.
\end{proof}

\begin{remark}
When $G \iso \Z_2 \o A_4$ then the curve has seven involutions as seen in Table 1. Thus, those curves belong to
the locus $\L_g$ of curves with extra involutions studied in \cite{GS}.
\end{remark}

Let $\ff: P^1 \to \bP^1$ be as above with monodromy group $A_4$. Then, the signature is $\s= ( \s_1, \s_2, \s_3)$,
where $\s_1$ is an involution and $\s_2, \s_3$ are elements of order 3 in $S_{12}$.  We choose branch points
$q_1=\infty$, $\,  q_2=6 i \sqrt{3}$, and $ q_3= - 6 i \sqrt{3}$, where $i^2=-1$.   We choose a coordinate $X$ in
$\bP^1$ such that $\ff(0)=\ff(\infty)=\ff(1)=\infty$.  Solving the corresponding system  of equations we find that
$$\ff(X)=\frac {X^{12} -33X^8 -33 X^4+1} {X^2 (X^4-1)^2}.$$
Thus, the  points in the fiber  of $q_1, q_2, q_3$ are the roots of the  following polynomials:
\begin{eqnarray*}
R(X)  & :=      X(X^4-1), \qquad \quad\\
S(X)  &  := \,    X^4-2i \sqrt{3} X^2 +1, \\
T(X) & := \, X^4+2i \sqrt{3} X^2 +1.
\end{eqnarray*}
Let $\l_i \in \bP^1 \setminus S$ be a branch point of the cover $\f: \X_g \to \bP^1$.  Thus, $\l_i^2+108\neq 0$.
Then points of $\ff^{-1}(\l_i)$ are roots of the polynomial
\begin{equation}  \label{G}
G_{i} (X)=X^{12} - \l_i X^{10} - 33 X^8 + 2 \l_i X^6 - 33 X^4 - \l_i X^2+1.
\end{equation}
$G_i(X)$ has distinct roots for $\l_i^2 \neq  108$.
\begin{remark}
The rational function $\phi (x)$  generates the fixed field of $A_4$ in $k(X)$. It was known to
Klein and has appeared many times in the literature since. Indeed, we picked the coordinate $X$ in
$\bP^1$ such that our expression of  $\phi (X)$ would be in this form.
\end{remark}
%
%*******************************
We now can compute the equation of the curve in all  cases 4-8 of Table 1.  If $\P \cap V= \emptyset $ then  the
equation of the curve is $Y^2=G(X)$ where
$$G(X)= \prod_{i=1}^{\delta } G_{\l_i} (X) $$
and $\delta = \frac {g+1} 6$; see Lemma~(4.1). If $\ff^{-1}(q_i)\subset \P$ then the polynomial corresponding  to
$q_i$ multiplies $G(X)$. Hence, we have the following:
\begin{lemma}
The equations of the curve $\X_g$ in each case are given by:
\end{lemma}
\begin{table}[ht]     \label{table}
\begin{center}
\renewcommand{\arraystretch}{1.24}
\begin{tabular}{||c|c|c||}
\hline \hline $G$  &  $\delta$  & Equation $Y^2= $   \\
\hline \hline $\Z_2\o A_4$ & $\frac {g+1} 6$ & $G(X)$ \\ $\Z_2\o A_4$ & $\frac {g-1} 6$ & $
(X^4+2i \sqrt{3}X^2 +1)\cdot  G(X)$ \\
$\Z_2\o A_4$ &$\frac {g-3} 6$& $ (X^8+14 X^4+1)  \cdot G(X)$   \\
$SL_2(3) $ &   $\frac {g-2} 6$ &  $X (X^4-1)  \cdot G(X) $ \\
$SL_2(3) $ &  $\frac {g-4} 6$ & $X (X^4-1)(X^4+2i \sqrt{3} X^2 +1)
\cdot G(X)$\\
$SL_2(3) $ &   $\frac {g-6} 6$ &  $X(X^4-1) (X^8+14X^4+1)\cdot G(X)$ \\
\hline \hline
\end{tabular} \end{center}
\caption{Hyperelliptic curves $\X_g$ with $\bAut(\X_g)$ isomorphic to $A_4$.}
\end{table}
%
%
%\vspace{-5mm}
%
\begin{remark} From the group theory viewpoint we have
$\s_1(X): X \to -X$ and $\s_2: X \to \frac {X-i} {X+i}$,
where $i^2=-1$. It is easily checked that $\s_1, \s_2$ generate $A_4$.
 Let $t\in \P$. The orbit of $A_4$ in $\P$ is
$$ t, \frac {t-i} {t+i}, -i \frac {t+1} {t-1}, \frac {t+i} {t-i}, -i
\frac {t-1} {t+1}, \frac 1 t, -t, - \frac {t-i} {t+i}, i \frac {t+1} {t-1}, - \frac {t+i} {t-i}, i \frac {t-1}
{t+1}, - \frac 1 t
$$
We label these points as $\a_1, \dots , \a_{12}$. Then, the polynomial $G_{i} (X)= \prod_{i=1}^{12} (x- \a_i)$ is
the polynomial in (\ref{G}) where $\l_i=   \frac {t^{12} -33t^8 -33 t^4+1} {t^2 (t^4-1)^2}$.
\end{remark}
Let  $\X_g$ be a given curve. When does $\X_g$ belong to one of the above cases?  Can we find a condition on the
coefficients of the curve such that $|\Aut(\X_g)|=24$? The following lemma determines a necessary condition that
$|\Aut(\X_g)|=24$.
\begin{lemma} Let $\X_g$ be a  curve with
$\bAut(\X_g)\iso A_4$. Then $I_4(\X_g)=0$.
%Moreover, if $g=1 \mod 6$ then  $I_2(\X_g)=0$ also.
\end{lemma}
\begin{proof}Computationally we show that  $I_4(G_i)=0$. Then, lemma follows
from properties of transvections.

\end{proof}

%****************************************************
\section{Applications, subvarieties $\L_g^G$ of dimension $\delta \leq 1$}
%***************************************************

In this section, we study in more detail  subvarieties $\L_g^G$ in Table 1, of dimension $\delta \leq 1$ when
$G\iso \Z_2\o A_4, SL_2(3)$.  We determine invariants which classify isomorphism  classes of such curves. These
invariants are used to prove that the field of definition of such curves is the same as the field of moduli. For
curves with  automorphism group $\Z_2 \o A_4$ this is a consequence of Theorem 4.2., in \cite{Sh5} and holds for
any genus. However, no results are know for curves with automorphism  group $SL_2(3)$. Proposition 5.2., addresses
this question for $g\leq 12$. From equations in Table 2 we get  the following lemma:
\begin{lemma} Let $\X_g$ be a hyperelliptic curve of genus $g \leq 12$. Then.

i) if $g=4$ then $I_2=I_4=I_4^{'}=I_6^{'}=0$

ii) if $g=5, 9,  12$ then  $I_4=I_6=0$

iii) if $g=7, 10$ then $I_2=I_4=I_4^{'}=I_6^\ast =0$

iv) if $g=8$ then $I_4=0$.
\end{lemma}
\noindent Then we define $\p(\X_g)$ as follows:
\begin{eqnarray*}
\p(\X_g):=(\p_1, \p_2)= \left\{ \aligned \v_1, \qquad \qquad \qquad \qquad \textit{  if  } g=4,   \\ (i_1, i_2),
\quad \textit{  if  } g=5, 9, \textit{  and   }   I_2\neq 0   \\ \v_2,   \quad \quad \textit{  if  } g=5, 9,
\textit{  and   }   I_2=0   \\  (j_1, j_2), \qquad \textit{  if  } g=7, \, \,  \textit{  and   }   I_3 \neq 0
\\ \v_3,  \qquad \quad \textit{  if  } g=7, \, \,  \textit{  and   }
I_3 = 0 \\ (i_1, i_3),  \quad \textit{  if  } g=8, 12,  \textit{  and }   I_2\neq 0  \\ \v_4,  \qquad \textit{  if
} g=8, 12,  \textit{ and   }   I_2 = 0  \\ (s_2, s_1),  \quad \textit{  if  } g=10, \, \, \textit{  and   }
I_{12}\neq 0     \\ \v_5,  \qquad \textit{  if  } g=10, \, \,  \textit{  and   } I_{12}= 0 \\ \endaligned \right.
\end{eqnarray*}
The following theorem uses these invariants to parametrize spaces $\L_g^G$, where  $G\iso \Z_2\o A_4, SL_2(3)$. In
the case that $\delta =1$ these spaces are  genus 0 curves.  This can be proved via Hurwitz spaces, as noticed in
Section 2.  However, the next theorem provides an algebraic  description of such spaces.
\begin{theorem}
Let $\X_g, \X_g^{'}$ be  genus $g\leq 12$  hyperelliptic curves with automorphism group $\Z_2 \o A_4, SL_2(3)$.
Then, $\X_g \iso \X_g^{'}$ if and only if $\p(\X_g) = \p(\X_g^{'})$.  Moreover, the moduli space $\L_g^G$ can be
parametrized as follows:
\begin{small}
\begin{eqnarray*} \label{eq}
\L_5^G:  &
 \p  =  \left(  \frac {49} {3630}\, \frac {(5\theta -484)^2}
{(5\theta +924)^2}  ,  \, \,  \frac {10} {27951} \, \frac  {\theta \,  (5\theta +30492)^2} {(5\theta +924)^3}
\right), \,  \textit{ and } \, \p= \frac {3^7 \cdot 5^3} {2^5 \cdot 7^2},
\, \,   \textit{ if } \, \, \t= - \frac {924} 5. \\ \\
\L_7^G: & \p=  \left( \frac 6 {245} \frac {(97\m+1606)^2 (87\m^2+528\m +2596)^2}  {(1093\m^3+49566\m^2-838068\m
+1549769)^2},
\right. \\
& \left. \frac {301158} {30625} \frac {(61\m^2-44\m - 6556)^2 (2021\m^2+1496\m -157476)}
{(1093\m^3+49566\m^2-838068\m
+1549769)^2} \right) \\
&  \textit{and    }   \quad    8000000\, \p^3 -
404568000  \, \p^2 -31666132872\, \p +308290455=0,  \\
&  \textit{if }  \qquad    1093\m^3+49566\m^2-838068\m +1549769)=0. \\ \\
\L_8^G: & \p= \left( \frac {49} {11236320} \frac {(279\m+15028)^2} {(7\m+884)^2}, \,   \frac 1 {1360026486} \frac
{\m
(3675\m+3321188)^2} {(7\m +884)^3} \right) \\
& \textit{ and  } \quad
 \p=\frac {2^3 \cdot 3^{11}\cdot 101^4} {5^3\cdot 7^4 \cdot 13^6},
 \textit{    if    } \, \t= - \frac {884} 7. \\ \\
\L_9^G: & \p=  \left(  \frac {605} {5633766}  \frac {(9\t -7200)^2} {(3\t +836)^2},  \,   \frac {90} {370680937}
\frac {\t
(157\t +79420)^2} {(3\t +836)^3} \right) \\
&  \textit{ and  } \quad \p=- \frac {2^9\cdot 5\cdot 11^2} {3^7},  \,   \textit{    if    }
\, \t= - \frac {836} 3. \\ \\
\L_{10}^G: &  \p=  \left( \frac {147} {90250} \frac {( 181\t + 1598)^2 \, (3813 \t^2 +15912 \t - 39236)^2}
{(251\t - 782)^2 \, (115 \t^2 - 68\t - 6596)^2}, \right. \\ & \left.  \frac {5007792000}{121} \frac {(7877 \t^2
+3128 \t -374884)^2 \, (115\t^2 - 68\t -6596)^2} {(251\t -782)^2 (181\t + 1598)^2 \, (3813 \t^2 + 15912
\t -39236)^2} \right) \\
 \textit{if } \, &
(251\t - 782)\, (115 \t^2 - 68\t - 6596)\, (181\t + 1598)\,
(3813 \t^2 + 15912\t -39236)=0 \\
 \textit{ then  } \, &
\p=  - \frac {950367275} {21168 \cdot Q} \cdot \frac {(402998158  \t^2-4363415636  \t+25170477
\t^3+13083554824)^4}
{(1268277  \t^2-5261568  \t+18129548)^2 \, (-287844+7959  \t^2-6575 6  \t)^2 } \cdot \\
& \hspace{-5mm} \frac 1 {228384659961\t^4-22196181318948\t^3+185588379432544\t^2-447275488903152\t
+658755318269936}
 \\ \\
\L_{12}^G: &
  \p=  \left( \frac 1 {268203000} \frac
{(6611\t - 501500 )^2} {(11\t + 1700)^2},  \,  \frac {56} {284015801875} \frac {\t \, (20933 \t - 5686500)^2}
{(11\t + 1700)^2}
\right)  \\
&  \textit{ and  }   \p=\frac {2 \cdot 3^3 \cdot 5 \cdot 41^4} {7^4 \cdot 11^2\cdot 17^2}, \quad \textit{    if
} \,
\, \m= - \frac {1700} {11},  \\ \\
\L_4^G: &    \p=  \frac {1764} {25}.
\end{eqnarray*}
\end{small}
\end{theorem}
\begin{proof}The proof of the theorem is computational.  Let $\X_g$ be a
hyperelliptic curve of $g \leq 12$ with  $\Aut(\X_g)\iso \Z_2 \o A_4$, or $SL_2(3)$.  From Lemma 2.1. we have
$g=4, 5, 7, 8, 9, 10, 12$.  If $g=4$ then the curve is isomorphic to
$$Y^2= X (3X^4+1) (3X^4+6X^2-1),$$
hence $\p(\X_4)=\frac {1764} {25}$.

For other cases we  compute $\p$ where the equation of the curve is given as in  Table 2.  If $g=5, 8, 9, 10, 12$,
we substitute $\t=\l^2$ and get expressions in (\ref{eq}).  For  $g=7, 10$  we substitute  $\m:=\l \sqrt{-3}$ and
get  $\p$ is as in (\ref{eq}).

\end{proof}

In each case one can find an equation for $\L_g^G$ by eliminating $\m$. Each nonsingular  point $\p \in \L_g^G$
correspond to a hyperelliptic curve $\X_g$ (up to isomorphism) with automorphism group $G$. One can show that
singular points of $\L_g^G$ corresponds  to hyperelliptic curves $\X_g$ such that $G$ is a proper subgroup of
$\Aut(\X_g)$.

\begin{example}
Let $g=5$. Then  eliminating $\m$ from equations of $\L_G^5$ we get
\begin{equation}\label{L_5}
4920750i_1^3  - 28224i_2^2 - 164025i_1^2  - 136080i_1i_2 + 672i_2 +1620i_1 -4=0
\end{equation}
There is only one singular point $\p=\left( 0, \frac 1 {84} \right)\in \L_G^5$. The genus 5 curve corresponding to
this point has equation
$$Y^2=X^{12}-\frac {484} 5 X^{10} - 33 X^8 + \frac {968} 5 X^6 - 33 X^4 -
\frac {484} 5 X^2 +1$$ One can show that this curve has automorphism group of order 48. It is the only genus 5
hyperelliptic curve (up to isomorphism) with automorphism group of order 48.
\end{example}

\subsection{Field of moduli versus field of definition}
%*********************************************************

Let $\X$ be a curve defined over $k$.  The {\bf field of moduli} of $\X$ is a subfield $F \subset k$ such that for
every automorphism $\sigma$ of $k$ $\X$ is isomorphic to   $\X^\sigma$ if and only if $ \sigma_F = id$. The field
of moduli is not necessary a field of definition.

In \cite{Sh5} we  conjectured that the field of moduli is the field of definition for all hyperelliptic curves
with extra automorphisms (i.e., automorphism group  of order $> 2$). Moreover, it is a corollary of Theorem 4.2.
in \cite{Sh5},  that field of moduli  is the field of definition for all hyperelliptic curves $\X_g$ which have
more than 3 involutions. As a consequence this is  the case for curves $\X$ with  $\Aut(\X)\iso  \Z_2 \o A_4$ (see
Table 1 for the number of involutions).   This is done in  \cite{Sh5} via dihedral involutions. Since every curve
with automorphism group $\Z_2 \o A_4$ has an extra involution then  its equation can be written as $$Y^2=F(X^2).$$
Then the field of moduli is determined  by the $g$-tuple of dihedral invariants $\u=(\u_1, \dots , \u_g)$. The
reduced  automorphism group of the curve has another involution if and only if
\begin{equation}\label{u}
2^{g+1} u_1^2 - 4 u_g^{g+1}=0
\end{equation}
In this case we give an equation of the curve in terms of the dihedral invariants $\u_1, \dots , \u_g$, see
\cite{Sh5} for details. We illustrate this method with the case $g=5$.

\begin{example} Let $g=5$ and $\X_5 \in \L_G^5$.  Then,
$$Y^2=X^{12} - \l X^{10} - 33 X^8 + 2 \l X^6 - 33 X^4 - \l X^2+1$$
Dihedral invariants of this curve are:
$$u_1=2\l^6 , u_2=-66 \l^4, u_3=-4\l^4, u_4= - 66\l^2, u_5= 2 \l^2, $$
see (\ref{u_i}) for their definitions with $n=2$. It can be easily checked that (\ref{u}) holds. Hence, the field
of moduli is the same as the field of definition from   results in \cite{Sh5}.
\end{example}

Let $\X$ be a curve which belongs to one of the spaces in (\ref{eq}).  Then the field of moduli is determined by
$\p(\X)=(\p_1, \p_2)$. The field of moduli is a field of definition if one can find a curve $\Y$ isomorphic to
$\X$ and with coefficients given as rational functions in $\p_1, \p_2$.

\begin{proposition}  %
Let  $\p \in \L_g^G$ and $\t$ the parameter of Theorem (5.1.). Then,
 $\X_g$ such that $\p =[\X_g]$ is isomorphic to

\begin{equation}  \label{prop}
\begin{split}
 \X_5:   \quad Y^2= &  \, \, M(X),  \\
\X_7:   \quad Y^2 = &   \, \, (3 X^4 + 6X^2 -1)   \\
 & \, \,  (27X^{12}-27\r X^{10}+297X^8 - 18X^6-99X^4+3\r X^2 +1),       \\
\X_8: \quad  Y^2= & \, \, X (\m X^4-1)\, M(X)\\
\X_9: \quad     Y^2 = & \, \, (\m^2\, X^8 + 14\m \, X^4 +1)\, M(X), \\
\X_{10}: \quad
Y^2 =  & \, \, X(3X^4+1)    (3 X^4 + 6X^2 -1)  \\
 & \, \, (27X^{12}-27\r X^{10}+297X^8 - 18X^6-99X^4+3\r X^2 +1),        \\
\X_{12}:\quad  Y^2= & \, \, X (\m X^4-1)(\m^2 X^8+\m X+1)\, M(X), \\
\end{split}
\end{equation}

where $M(X) :=\t^3 \, X^{12} - \t^3\, X^{10} - 33 \, \t^2\, X^8 +2\, \t^2 \, X^6 - 33\,  \t\, X^4 - \t \, X^2 +1.$
\end{proposition}

\begin{proof}
For the computationally minded reader, simply compute the invariants in each case. These invariants are the same
as given in (\ref{eq}). Thus, the curve  corresponds to $\p$.
 However, it is important to know explicitly
the isomorphism between curves in (\ref{eq}) and curves in (\ref{prop}).

Fix $\p \in \L_g$ as in (\ref{eq}).  Let $g= 5,8,9,12$. Then the  curve $\X_g$   given in Table 2 corresponds to
$\p$ (in all cases) since that is how we computed $\p$.   Let  $X=\sqrt{\l}\, \,  X$ and $\t =\l^2$ and we get
equations as in (\ref{prop}).  Since, $\t=\l^2 $ in all these cases of Theorem 5.1., then the proof is complete.

Let $g=7$. The equation of the curve is
$$Y^2=(X^4+ 2 \sqrt{-3}  X^2 +1) G_1(X)$$
Let $\r:= \l \sqrt{-3}$ and perform the transformation $X \to (-3)^{\frac 1 4} X$. Hence, the curve is isomorphic
to $\X_7$ in (\ref{prop}). The case $g=10$ goes the same way as $g=7$.
\end{proof}

\begin{proposition}
Let $\X$ be a hyperelliptic curve of genus $ g \leq 12$  such that $\bAut(\X)\iso A_4$. Then, the field of
definition of $\X$ is the same as the field of moduli.  Moreover, the equation of the curve in terms of its
invariants is given in (\ref{prop}).
\end{proposition}

\begin{proof}
In each case of (\ref{eq}) it is easily shown that $\t$ is a generator of $k(\p_1, \p_2)$  (i.e., $\t$ can be
expressed as a rational function in terms of $\p_1, \p_2$).

\end{proof}

There is no particular reason why we focused on $g\leq 12$. Indeed the same technique can  be used for higher
genus. However, there is no obvious way to generalize Proposition 5.2. to any genus $g$ in the case when the group
is $SL_2(3)$.

\end{document}